\documentclass[11pt]{article}

\usepackage[utf8]{inputenc}
\usepackage[margin=1.05in]{geometry}
\usepackage{amsmath,amssymb,amsthm}
\usepackage{mathtools}
\usepackage{array}
\usepackage{booktabs}
\newcolumntype{P}[1]{>{\raggedright\arraybackslash}p{#1}}
\usepackage{enumitem}
\usepackage{xcolor}
\usepackage{mdframed}
\usepackage[protrusion=true,expansion=false]{microtype}
\usepackage[hidelinks,hypertexnames=false]{hyperref}

\theoremstyle{plain}
\newtheorem{theorem}{Theorem}[section]
\newtheorem{lemma}[theorem]{Lemma}
\newtheorem{proposition}[theorem]{Proposition}
\newtheorem{corollary}[theorem]{Corollary}
\newtheorem{quotedthm}[theorem]{Theorem}

\theoremstyle{definition}
\newtheorem{definition}[theorem]{Definition}

\theoremstyle{remark}
\newtheorem{remark}[theorem]{Remark}

\newmdtheoremenv[
  linewidth=1pt,
  linecolor=black,
  backgroundcolor=black!4,
  innertopmargin=8pt,
  innerbottommargin=8pt,
  innerleftmargin=10pt,
  innerrightmargin=10pt,
  roundcorner=3pt
]{unproved}[theorem]{Unproved quantitative hypothesis}

\DeclareMathOperator{\Tr}{Tr}
\newcommand{\HS}{\mathrm{HS}}
\newcommand{\RR}{\mathbb{R}}
\newcommand{\BG}{\mathsf{G}}
\newcommand{\dd}{\,d}

\title{Entropy-type traces and moment expansions for the\\
sine-kernel time--band limiting operator}
\author{Ahmadreza Azimifard}
\date{August 2026}

\hypersetup{%
  pdftitle={Entropy-type traces and moment expansions for the sine-kernel
time-band limiting operator},
  pdfauthor={Ahmadreza Azimifard},
  pdfsubject={Spectral theory of the sine-kernel time-band limiting
operator; entropy-type traces and moment expansions},
  pdfkeywords={time-band limiting operator, sine kernel, prolate
spheroidal wave functions, Landau-Widom asymptotics, entropy trace,
moment expansion, eigenvalue counting function}}

\begin{document}
\maketitle

\begin{abstract}
We study the sine-kernel time--band limiting operator $S_c$ on
$L^2(0,c)$ through the entropy-type trace
$\Tr\varphi_u(S_c)$, where
$\varphi_u(x)=\log(1+(e^u-1)x)-ux$.  A Fourier factorization gives
$S_c=A_c^*A_c$ and $\Tr S_c=c$.  We derive an exact identity for
$\Tr(S_c-S_c^2)$ and the lower bound
$\Tr(S_c-S_c^2)\ge\pi^{-2}\log c-\pi^{-3}$, which in turn yields
$\Tr\varphi_u(S_c)\ge u\log c/(4\pi^2)$.  We also obtain a positive
moment expansion
$\Tr\varphi_u(S_c)=\sum_{n\ge1}A_n(u)T_n(c)$ and, for each fixed $n$,
the Landau--Widom asymptotic
$T_n(c)=\log c/(\pi^2n)+o_n(\log c)$.

Three natural uniformity hypotheses, denoted (LC), (MT), and (QM), lead
to a quadratic lower bound of order $u^2\log c$ on logarithmically
growing windows.  We prove the implications among these hypotheses and
the resulting bounds, with explicit constants, and explain why the
fixed-index asymptotic alone does not provide the required uniformity.
The hypotheses themselves remain open.  Finally, a signed
growing-parameter sine-kernel determinant theorem from a companion
paper gives the quadratic lower bound independently of (LC), (MT), and
(QM).  This separates the conclusions available from elementary
operator methods from those that use Riemann--Hilbert asymptotics.


\end{abstract}

\section{Setting and main question}
\label{sec:setting}

\begin{definition}[Sine kernel and the operator $S_c$]
\label{def:kernel}
Let
\begin{equation}
  K(s)=\frac{\sin(\pi s)}{\pi s},\qquad K(0)=1 ,
  \label{eq:kernel}
\end{equation}
which is real analytic and even on $\RR$, with $|K|\le1$. For $c>0$ let $S_c$
be the integral operator on $L^2(0,c)$ with kernel
\begin{equation}
  S_c(x,y)=K(x-y),\qquad 0<x,y<c .
  \label{eq:Sc}
\end{equation}
\end{definition}

The kernel is bounded on a set of finite measure, hence square integrable on
$(0,c)^2$; therefore $S_c$ is a Hilbert--Schmidt, in particular compact,
operator. It is symmetric and real, hence self-adjoint. We write
$\lambda_1(c)\ge\lambda_2(c)\ge\cdots$ for its eigenvalues, repeated according
to multiplicity; Theorem~\ref{thm:factorization} shows that all of them lie in
$[0,1]$, and Remark~\ref{rem:strict} shows that in fact
$0<\lambda_j(c)<1$ for every $j$.

\begin{definition}[The entropy-type function $\varphi_u$]
\label{def:phi}
For $u\ge0$ and $x\in[0,1]$ set
\begin{equation}
  \varphi_u(x)=\log\bigl(1+(e^u-1)x\bigr)-ux .
  \label{eq:phi}
\end{equation}
For each fixed $u$ the function $\varphi_u$ is continuous on $[0,1]$ with
$\varphi_u(0)=\varphi_u(1)=0$, and $\varphi_0\equiv0$.
\end{definition}

\begin{definition}[Moments and their weights]
\label{def:moments}
For integers $n\ge1$ and $u\ge0$ put
\begin{equation}
  T_n(c)=\Tr\bigl(S_c(I-S_c)^n\bigr),
  \qquad
  A_n(u)=\int_0^u\bigl(1-e^{-t}\bigr)^n\dd t .
  \label{eq:TnAn}
\end{equation}
\end{definition}

Since $0\le 1-e^{-t}<1$, each $A_n(u)$ is finite, nonnegative, and
nonincreasing in $n$; the trace defining $T_n(c)$ is finite and nonnegative by
Proposition~\ref{prop:Tnbounds}.

\paragraph{Main question.}
The estimates below are organized around the inequality
\begin{equation}
  \Tr\varphi_u(S_c)\ \ge\ \kappa_T\,u^2\log c,
  \qquad L_T\le u\le\alpha\log c ,
  \tag{$\mathrm{T}_\alpha$}
  \label{eq:target}
\end{equation}
for suitable positive constants $\kappa_T$, $L_T$, and $\alpha$, and
all sufficiently large $c$.  We consider two complementary approaches.

\begin{enumerate}[label=\textup{(\Roman*)},leftmargin=2.6em]
\item \emph{Operator and moment methods.}  Sections
\ref{sec:factorization}--\ref{sec:qm} develop, from first principles, the
factorization, variance, logistic-window, and moment-expansion machinery
for $S_c$.  This machinery yields unconditionally the \emph{linear}
bound \eqref{eq:E3}, and it yields the quadratic target
\eqref{eq:target} \emph{conditionally} on any one of three hypotheses
(LC), (MT), (QM), by the proved implications of
Propositions~\ref{prop:R1}, \ref{prop:A4} and Theorem~\ref{thm:P5}.
The hypotheses themselves remain open; Section~\ref{sec:qm} explains
why the fixed-index Landau--Widom asymptotic does not supply the needed
uniformity.
\item \emph{The determinant method.}  Section~\ref{sec:determinant}
uses a signed growing-parameter sine-kernel determinant estimate from
the companion paper \cite{CompanionC} and derives \eqref{eq:target}
with $L_T=2$,
$\kappa_T=1/(4\pi^2)$ and every $\alpha>0$.  This argument uses none of
(LC), (MT), or (QM); its nonlinear steepest-descent input is proved in
\cite{CompanionC}.
\end{enumerate}

The first approach gives structural information that is not contained
in the determinant asymptotic, including the variance identity
\eqref{eq:A2}, the positive expansion \eqref{eq:P1}, and the counting
identity of Lemma~\ref{lem:count}.  The two approaches can therefore be
read independently.  Their precise assumptions are summarized in
Section~\ref{sec:cert}.

\begin{remark}[Scope]\label{rem:noclaim}
The counting results considered here are conditional consequences of
(LC) or (MT).  The related bridge problem asks for a lower bound of order
$L\log(\alpha c/L)$, $L=\log\frac1\delta$, for
$\#\{j:\delta<\lambda_j(c)\le\tfrac12\}$ uniformly on a moving threshold
range and is treated in \cite{CompanionC}.
\end{remark}

\paragraph{Fourier normalization.}
Throughout, the Fourier transform on $L^2(\RR)$ is
\begin{equation}
  (\mathcal{F}h)(\xi)=\widehat h(\xi)
  =\frac1{\sqrt{2\pi}}\int_{\RR}e^{-ix\xi}h(x)\dd x ,
  \label{eq:ft}
\end{equation}
which is unitary on $L^2(\RR)$ (Plancherel), with inverse
$(\mathcal{F}^{-1}g)(x)=(2\pi)^{-1/2}\int_{\RR}e^{ix\xi}g(\xi)\dd\xi$. Inner
products are linear in the first and conjugate linear in the second
argument.

\section{Fourier factorization, contraction, and the trace}
\label{sec:factorization}

\begin{definition}
\label{def:Ac}
Let $A_c\colon L^2(0,c)\to L^2(-\pi,\pi)$ be defined by
\begin{equation}
  (A_cf)(\xi)=\frac1{\sqrt{2\pi}}\int_0^ce^{-ix\xi}f(x)\dd x,
  \qquad-\pi<\xi<\pi .
  \label{eq:Ac}
\end{equation}
\end{definition}

\begin{theorem}[Factorization, contraction, trace]
\label{thm:factorization}
The operator $A_c$ is a Hilbert--Schmidt contraction, and
\begin{equation}
  S_c=A_c^*A_c .
  \label{eq:factor}
\end{equation}
Consequently $0\le S_c\le I$, so $\sigma(S_c)\subset[0,1]$; moreover
$\|A_c\|_{\HS}^2=c$, the operator $S_c$ is trace class, and
\begin{equation}
  \Tr S_c=c .
  \tag{A1}
  \label{eq:A1}
\end{equation}
\end{theorem}

\begin{proof}
\emph{Step 1: the adjoint.}
Let $f\in L^2(0,c)$ and $g\in L^2(-\pi,\pi)$. Both intervals have finite
measure, so by Cauchy--Schwarz $f\in L^1(0,c)$ and $g\in L^1(-\pi,\pi)$, and
the function $(x,\xi)\mapsto e^{-ix\xi}f(x)\overline{g(\xi)}$ is
absolutely integrable on the rectangle $(0,c)\times(-\pi,\pi)$, with
$\int\!\!\int|f(x)|\,|g(\xi)|\dd x\dd\xi
=\|f\|_{L^1}\|g\|_{L^1}<\infty$.
Fubini's theorem therefore applies and gives
\[
  \langle A_cf,g\rangle
  =\int_{-\pi}^{\pi}\!\Bigl(\frac1{\sqrt{2\pi}}\int_0^ce^{-ix\xi}f(x)\dd x\Bigr)
   \overline{g(\xi)}\dd\xi
  =\int_0^cf(x)\,
   \overline{\Bigl(\frac1{\sqrt{2\pi}}\int_{-\pi}^{\pi}e^{ix\xi}g(\xi)\dd\xi\Bigr)}
   \dd x ,
\]
where we used $\overline{e^{ix\xi}g(\xi)}=e^{-ix\xi}\overline{g(\xi)}$. Hence
\begin{equation}
  (A_c^*g)(x)=\frac1{\sqrt{2\pi}}\int_{-\pi}^{\pi}e^{ix\xi}g(\xi)\dd\xi,
  \qquad 0<x<c .
  \label{eq:Acstar}
\end{equation}

\emph{Step 2: the composition.}
Fix $f\in L^2(0,c)$ and $x\in(0,c)$. Applying \eqref{eq:Acstar} to $g=A_cf$,
\[
  (A_c^*A_cf)(x)
  =\frac1{2\pi}\int_{-\pi}^{\pi}e^{ix\xi}
    \Bigl(\int_0^ce^{-iy\xi}f(y)\dd y\Bigr)\dd\xi .
\]
The integrand is bounded in modulus by $|f(y)|$, which is integrable on the
rectangle $(0,c)\times(-\pi,\pi)$ of finite measure; Fubini's theorem permits
exchanging the order of integration, giving
\[
  (A_c^*A_cf)(x)
  =\int_0^c\Bigl(\frac1{2\pi}\int_{-\pi}^{\pi}e^{i(x-y)\xi}\dd\xi\Bigr)f(y)\dd y .
\]
For $s\ne0$ the inner integral equals
$\frac{1}{2\pi}\cdot\frac{e^{i\pi s}-e^{-i\pi s}}{is}
=\frac{\sin(\pi s)}{\pi s}=K(s)$, and for $s=0$ it equals
$\frac{1}{2\pi}\cdot2\pi=1=K(0)$. Thus the inner integral is exactly $K(x-y)$
for all $s=x-y\in\RR$, and $(A_c^*A_cf)(x)=\int_0^cK(x-y)f(y)\dd y=(S_cf)(x)$.
This proves \eqref{eq:factor}.

\emph{Step 3: contraction.}
Let $E\colon L^2(0,c)\to L^2(\RR)$ be extension by zero and
$R\colon L^2(\RR)\to L^2(-\pi,\pi)$ restriction. Then $E$ is an isometry and
$R$ is a contraction, and \eqref{eq:Ac} says precisely that
$A_c=R\,\mathcal{F}E$ with $\mathcal{F}$ as in \eqref{eq:ft}: indeed, for
$f\in L^2(0,c)$ the zero extension $Ef$ satisfies
$\widehat{Ef}(\xi)=(2\pi)^{-1/2}\int_0^ce^{-ix\xi}f(x)\dd x$. Since
$\mathcal{F}$ is unitary,
\[
  \|A_cf\|_{L^2(-\pi,\pi)}\le\|\mathcal{F}Ef\|_{L^2(\RR)}
  =\|Ef\|_{L^2(\RR)}=\|f\|_{L^2(0,c)} ,
\]
so $\|A_c\|\le1$. Consequently, for every $f$,
\[
  0\le\langle S_cf,f\rangle=\|A_cf\|^2\le\|f\|^2 ,
\]
i.e.\ $0\le S_c\le I$; since $S_c$ is self-adjoint and compact, its spectrum,
and in particular every eigenvalue, lies in $[0,1]$.

\emph{Step 4: Hilbert--Schmidt norm and the trace.}
$A_c$ is an integral operator with kernel
$a(\xi,x)=(2\pi)^{-1/2}e^{-ix\xi}$ on $(-\pi,\pi)\times(0,c)$, so
\begin{equation}
  \|A_c\|_{\HS}^2
  =\int_{-\pi}^{\pi}\!\int_0^c|a(\xi,x)|^2\dd x\dd\xi
  =\int_{-\pi}^{\pi}\!\int_0^c\frac{\dd x\dd\xi}{2\pi}
  =\frac{2\pi\cdot c}{2\pi}=c<\infty .
  \label{eq:HS}
\end{equation}
Thus $A_c$ is Hilbert--Schmidt. A product of two Hilbert--Schmidt operators is
trace class, so $S_c=A_c^*A_c$ is trace class; and for any Hilbert--Schmidt
operator $B$ one has $\Tr(B^*B)=\|B\|_{\HS}^2$. With $B=A_c$ and
\eqref{eq:HS} this yields $\Tr S_c=c$, which is \eqref{eq:A1}.
\end{proof}

\begin{remark}[Consistency check]
\label{rem:diagonal}
Since $S_c$ is a positive trace-class integral operator with continuous kernel,
Mercer's theorem gives $\Tr S_c=\int_0^cS_c(x,x)\dd x=\int_0^cK(0)\dd x=c$,
in agreement with \eqref{eq:A1}.
\end{remark}

\begin{remark}[Strictness of the spectral inclusion]
\label{rem:strict}
No eigenvalue equals $0$ or $1$. If $S_cf=0$ then $\|A_cf\|=0$, so
$\widehat{Ef}$ vanishes on $(-\pi,\pi)$; but $Ef$ has compact support, so
$\widehat{Ef}$ extends to an entire function of exponential type, and a
nontrivial such function cannot vanish on an interval. Hence $f=0$. If
$S_cf=f$ with $\|f\|=1$, then $\|A_cf\|=\|Ef\|$, so by Plancherel
$\widehat{Ef}$ vanishes outside $(-\pi,\pi)$; then $Ef$ and $\widehat{Ef}$ both
have compact support, which forces $Ef=0$ by the same analyticity argument
applied to $\widehat{Ef}$. Hence $0<\lambda_j(c)<1$ for all $j$. This remark is
not used anywhere below --- in particular the finiteness of the counting
function of Section~\ref{sec:logistic} follows already from trace-class
membership --- and is recorded only because it shows that the logistic
parametrization $\lambda=q(t)$, $t\in\RR$, of Section~\ref{sec:logistic}
reaches every eigenvalue of $S_c$.
\end{remark}

\section{Elementary bounds for \texorpdfstring{$\varphi_u$}{phi\_u}}
\label{sec:elementary}

\begin{lemma}[Exact derivative and the quadratic minorant]
\label{lem:E1}
For $x\in[0,1]$ and $u\ge0$,
\begin{equation}
  \partial_u\varphi_u(x)
  =\frac{(e^u-1)x(1-x)}{1+(e^u-1)x}
  =x(1-x)\,\frac{1-e^{-u}}{e^{-u}+(1-e^{-u})x} .
  \label{eq:dphi}
\end{equation}
Consequently
\begin{equation}
  \partial_u\varphi_u(x)\ \ge\ (1-e^{-u})\,x(1-x),
  \label{eq:dphilower}
\end{equation}
and, integrating,
\begin{equation}
  \varphi_u(x)\ \ge\ \bigl(u-1+e^{-u}\bigr)\,x(1-x) .
  \tag{E1}
  \label{eq:E1}
\end{equation}
In particular $\varphi_u\ge0$ on $[0,1]$ for every $u\ge0$.
\end{lemma}

\begin{proof}
Fix $x\in[0,1]$. The map $u\mapsto1+(e^u-1)x$ is smooth and satisfies
$1+(e^u-1)x\ge1>0$ for $x\in[0,1]$, $u\ge0$, because it is the convex
combination $(1-x)\cdot1+x\cdot e^u$ of $1$ and $e^u\ge1$. Differentiating
\eqref{eq:phi} in $u$,
\[
  \partial_u\varphi_u(x)
  =\frac{e^ux}{1+(e^u-1)x}-x
  =\frac{e^ux-x\bigl(1+(e^u-1)x\bigr)}{1+(e^u-1)x} .
\]
The numerator simplifies to
$x\bigl(e^u-1-(e^u-1)x\bigr)=(e^u-1)x(1-x)$, so
\[
  \partial_u\varphi_u(x)=\frac{(e^u-1)x(1-x)}{1+(e^u-1)x} .
\]
Multiplying numerator and denominator by $e^{-u}>0$ and using
$e^{-u}(e^u-1)=1-e^{-u}$ and
$e^{-u}\bigl(1+(e^u-1)x\bigr)=e^{-u}+(1-e^{-u})x$ gives the second expression
in \eqref{eq:dphi}.

For $x\in[0,1]$ the denominator of that second expression is a convex
combination of $1$ and $x$ with weights $e^{-u}$ and $1-e^{-u}$, hence
\[
  0<e^{-u}+(1-e^{-u})x\le e^{-u}\cdot 1+(1-e^{-u})\cdot 1=1 .
\]
Since the numerator $x(1-x)(1-e^{-u})$ is nonnegative on $[0,1]$, dividing by a
quantity that is at most $1$ can only increase it; this is
\eqref{eq:dphilower}.

Finally $\varphi_0(x)=\log1-0=0$, and $u\mapsto\varphi_u(x)$ is continuously
differentiable on $[0,\infty)$, so by the fundamental theorem of calculus and
\eqref{eq:dphilower},
\[
  \varphi_u(x)=\int_0^u\partial_t\varphi_t(x)\dd t
  \ \ge\ x(1-x)\int_0^u\bigl(1-e^{-t}\bigr)\dd t
  =\bigl(u-1+e^{-u}\bigr)x(1-x) .
\]
The factor $u-1+e^{-u}$ is nonnegative for $u\ge0$ (it vanishes at $u=0$ and
has derivative $1-e^{-u}\ge0$), and $x(1-x)\ge0$ on $[0,1]$; hence
$\varphi_u\ge0$ there.
\end{proof}

\begin{lemma}[Linear majorant and trace finiteness]
\label{lem:upper}
For $x\in[0,1]$ and $u\ge0$,
\begin{equation}
  0\ \le\ \varphi_u(x)\ \le\ \bigl(e^u-1-u\bigr)x .
  \label{eq:upper}
\end{equation}
Consequently $\varphi_u(S_c)$ is a positive trace-class operator with
\begin{equation}
  0\ \le\ \Tr\varphi_u(S_c)\ \le\ \bigl(e^u-1-u\bigr)c<\infty .
  \label{eq:tracefinite}
\end{equation}
\end{lemma}

\begin{proof}
The left inequality is the last assertion of Lemma~\ref{lem:E1}. For the right
one, apply $\log(1+y)\le y$, valid for $y>-1$, with $y=(e^u-1)x\ge0$:
\[
  \varphi_u(x)=\log\bigl(1+(e^u-1)x\bigr)-ux\le(e^u-1)x-ux=(e^u-1-u)x .
\]
Both bounds hold on the spectrum $\sigma(S_c)\subset[0,1]$, so by the
functional calculus for the bounded self-adjoint operator $S_c$,
\[
  0\le\varphi_u(S_c)\le(e^u-1-u)S_c
\]
in the operator order. A positive operator dominated by a trace-class operator
is trace class with a dominated trace; since $\Tr S_c=c$ by \eqref{eq:A1},
\eqref{eq:tracefinite} follows.
\end{proof}

\section{The variance identity and an explicit logarithmic
 lower bound}
\label{sec:variance}

\begin{definition}
\label{def:V}
Put $V_c=\Tr\bigl(S_c-S_c^2\bigr)$.
\end{definition}

Both $S_c$ and $S_c^2$ are trace class ($S_c$ by Theorem~\ref{thm:factorization},
$S_c^2$ as a product of trace-class operators), so $V_c$ is a well-defined real
number; and $0\le S_c\le I$ gives $S_c^2\le S_c$, so $V_c\ge0$.

\begin{lemma}[Normalization of the kernel]
\label{lem:plancherel}
$\displaystyle\int_{\RR}K(s)^2\dd s=1$.
\end{lemma}

\begin{proof}
Let $h=\chi_{(-\pi,\pi)}\in L^2(\RR)$. By the computation in Step~2 of the
proof of Theorem~\ref{thm:factorization},
\[
  (\mathcal{F}^{-1}h)(s)=\frac1{\sqrt{2\pi}}\int_{-\pi}^{\pi}e^{is\xi}\dd\xi
  =\sqrt{2\pi}\cdot\frac1{2\pi}\int_{-\pi}^{\pi}e^{is\xi}\dd\xi
  =\sqrt{2\pi}\,K(s)
\]
for all $s\in\RR$, including $s=0$. Since $\mathcal{F}^{-1}$ is unitary on
$L^2(\RR)$,
\[
  2\pi\int_{\RR}K(s)^2\dd s=\|\mathcal{F}^{-1}h\|_{L^2}^2=\|h\|_{L^2}^2=2\pi ,
\]
whence the claim. (In particular $K\in L^2(\RR)$.)
\end{proof}

\begin{theorem}[Variance identity]
\label{thm:A2}
For every $c>0$,
\begin{equation}
  V_c
  =\int_0^c\!\!\int_{\RR\setminus(0,c)}K(x-y)^2\dd y\dd x
  =2\int_0^\infty\min\{c,s\}\,K(s)^2\dd s ,
  \label{eq:varid}
\end{equation}
and therefore
\begin{equation}
  V_c=\frac2{\pi^2}\int_0^c\frac{\sin^2(\pi s)}{s}\dd s
     +\frac{2c}{\pi^2}\int_c^\infty\frac{\sin^2(\pi s)}{s^2}\dd s .
  \tag{A2}
  \label{eq:A2}
\end{equation}
\end{theorem}

\begin{proof}
\emph{Step 1: the trace of $S_c^2$.}
$S_c$ is Hilbert--Schmidt with kernel $K(x-y)$ on $(0,c)^2$, and for a
self-adjoint Hilbert--Schmidt operator $B$ with kernel $b$ one has
$\Tr(B^2)=\|B\|_{\HS}^2=\int\!\!\int|b|^2$. Hence
\[
  \Tr S_c^2=\int_0^c\!\!\int_0^cK(x-y)^2\dd y\dd x ,
\]
the double integral being finite because $|K|\le1$ on a set of finite measure.

\emph{Step 2: subtraction.}
By Lemma~\ref{lem:plancherel}, for each fixed $x$ we have
$\int_{\RR}K(x-y)^2\dd y=\int_{\RR}K(s)^2\dd s=1$ (translation invariance of
Lebesgue measure), so
$c=\int_0^c\bigl(\int_{\RR}K(x-y)^2\dd y\bigr)\dd x$.
All the integrands are nonnegative and the inner integrals are finite, so
subtracting Step~1 from \eqref{eq:A1} splits the inner integral over $\RR$
into the part over $(0,c)$ and the part over its complement:
\[
  V_c=\Tr S_c-\Tr S_c^2
  =\int_0^c\!\!\int_{\RR\setminus(0,c)}K(x-y)^2\dd y\dd x .
\]
This is the first equality in \eqref{eq:varid}.

\emph{Step 3: reflection and change of variables.}
Fix $x\in(0,c)$ and substitute $s=y-x$ in the inner integral. As $y$ runs over
$\RR\setminus(0,c)$, $s$ runs over $\RR\setminus(-x,\,c-x)$, i.e.\ over
$(-\infty,-x]\cup[c-x,\infty)$. Since $K$ is even, $K^2$ is even, and
$\int_{-\infty}^{-x}K(s)^2\dd s=\int_x^{\infty}K(s)^2\dd s$ by the reflection
$s\mapsto-s$. Writing $G(a)=\int_a^{\infty}K(s)^2\dd s$ for $a\ge0$, we obtain
\[
  V_c=\int_0^c\bigl[G(x)+G(c-x)\bigr]\dd x=2\int_0^cG(x)\dd x ,
\]
the last step by the measure-preserving change of variables $x\mapsto c-x$ on
$(0,c)$.

\emph{Step 4: Tonelli.}
The function $(x,s)\mapsto K(s)^2\chi_{\{0<x<c\}}\chi_{\{s>x\}}$ is
nonnegative and measurable on $(0,c)\times(0,\infty)$, so Tonelli's theorem
allows the order of integration to be exchanged:
\[
  2\int_0^cG(x)\dd x
  =2\int_0^c\!\!\int_x^{\infty}K(s)^2\dd s\dd x
  =2\int_0^{\infty}K(s)^2\,\bigl|\{x\in(0,c):x<s\}\bigr|\dd s .
\]
Since $|\{x\in(0,c):x<s\}|=\min\{c,s\}$, the right-hand side equals
$2\int_0^{\infty}\min\{c,s\}K(s)^2\dd s$, which is the second equality in
\eqref{eq:varid}.

\emph{Step 5: explicit form.}
Split at $s=c$ and insert $K(s)^2=\sin^2(\pi s)/(\pi^2s^2)$:
\[
  2\int_0^c s\,\frac{\sin^2(\pi s)}{\pi^2s^2}\dd s
  +2\int_c^{\infty}c\,\frac{\sin^2(\pi s)}{\pi^2s^2}\dd s
  =\frac2{\pi^2}\int_0^c\frac{\sin^2(\pi s)}{s}\dd s
  +\frac{2c}{\pi^2}\int_c^{\infty}\frac{\sin^2(\pi s)}{s^2}\dd s .
\]
Both integrals converge: near $s=0$ one has
$\sin^2(\pi s)/s\le\pi^2s$, and the tail integrand is at most $s^{-2}$. This is
\eqref{eq:A2}.
\end{proof}

\begin{lemma}[An oscillatory integral bound]
\label{lem:cos}
For every $c\ge1$,
\begin{equation}
  \Bigl|\int_1^c\frac{\cos(2\pi s)}{s}\dd s\Bigr|\le\frac1\pi .
  \label{eq:cosbound}
\end{equation}
\end{lemma}

\begin{proof}
Integrate by parts with $u=1/s$ and $dv=\cos(2\pi s)\dd s$, so that
$v=\sin(2\pi s)/(2\pi)$:
\[
  \int_1^c\frac{\cos(2\pi s)}{s}\dd s
  =\Bigl[\frac{\sin(2\pi s)}{2\pi s}\Bigr]_1^c
   +\frac1{2\pi}\int_1^c\frac{\sin(2\pi s)}{s^2}\dd s .
\]
Since $\sin(2\pi)=0$, the boundary term equals $\sin(2\pi c)/(2\pi c)$, whose
modulus is at most $1/(2\pi c)\le1/(2\pi)$ for $c\ge1$. The remaining term is
bounded by
\[
  \frac1{2\pi}\int_1^c\frac{\dd s}{s^2}\le\frac1{2\pi}\int_1^{\infty}\frac{\dd s}{s^2}
  =\frac1{2\pi} .
\]
Adding the two bounds gives \eqref{eq:cosbound}.
\end{proof}

\begin{theorem}[Explicit logarithmic lower bound for the variance]
\label{thm:A3}
For $c\ge1$,
\begin{equation}
  V_c\ \ge\ \frac1{\pi^2}\log c-\frac1{\pi^3} ,
  \label{eq:Vlower}
\end{equation}
and consequently
\begin{equation}
  V_c\ \ge\ \frac1{2\pi^2}\log c
  \qquad\bigl(c\ge e^{2/\pi}\bigr).
  \tag{A3}
  \label{eq:A3}
\end{equation}
Moreover
\begin{equation}
  V_c=\frac1{\pi^2}\log c+O(1)\qquad(c\ge1).
  \tag{E2}
  \label{eq:E2}
\end{equation}
\end{theorem}

\begin{proof}
\emph{Lower bound.}
Both terms on the right of \eqref{eq:A2} are nonnegative, so discarding the
tail term and restricting the first integral to $(1,c)$,
\[
  V_c\ \ge\ \frac2{\pi^2}\int_1^c\frac{\sin^2(\pi s)}{s}\dd s .
\]
Using $\sin^2(\pi s)=\tfrac12\bigl(1-\cos(2\pi s)\bigr)$,
\[
  \int_1^c\frac{\sin^2(\pi s)}{s}\dd s
  =\frac12\int_1^c\frac{\dd s}{s}-\frac12\int_1^c\frac{\cos(2\pi s)}{s}\dd s
  \ \ge\ \frac12\log c-\frac1{2\pi} ,
\]
by Lemma~\ref{lem:cos}. Multiplying by $2/\pi^2$ gives \eqref{eq:Vlower}.
For \eqref{eq:A3} it suffices that
$\tfrac1{2\pi^2}\log c\ge\tfrac1{\pi^3}$, i.e.\ $\log c\ge2/\pi$, i.e.\
$c\ge e^{2/\pi}$; note $e^{2/\pi}>1$, so \eqref{eq:Vlower} is available.

\emph{Upper bound.}
For the first term of \eqref{eq:A2}, split at $s=1$. On $(0,1)$ use
$\sin^2(\pi s)/s\le\pi^2s$, giving
$\int_0^1\sin^2(\pi s)s^{-1}\dd s\le\pi^2/2$. On $(1,c)$ use the same identity
and Lemma~\ref{lem:cos} in the opposite direction:
$\int_1^c\sin^2(\pi s)s^{-1}\dd s\le\tfrac12\log c+\tfrac1{2\pi}$. For the
second term of \eqref{eq:A2}, $\sin^2\le1$ gives
$\frac{2c}{\pi^2}\int_c^{\infty}s^{-2}\dd s=\frac{2}{\pi^2}$. Altogether
\[
  \frac1{\pi^2}\log c-\frac1{\pi^3}
  \ \le\ V_c\ \le\
  \frac1{\pi^2}\log c+\frac1{\pi^3}+1+\frac2{\pi^2} ,
\]
which is \eqref{eq:E2} with an explicit $O(1)$.
\end{proof}

\begin{corollary}[Unconditional linear-in-$u$ lower bound]
\label{cor:E3}
For $c\ge e^{2/\pi}$ and $u\ge2$,
\begin{equation}
  \Tr\varphi_u(S_c)\ \ge\ \frac{u}{4\pi^2}\log c .
  \tag{E3}
  \label{eq:E3}
\end{equation}
\end{corollary}

\begin{proof}
First, $u-1+e^{-u}\ge u/2$ for $u\ge2$: the function
$g(u)=u/2-1+e^{-u}$ satisfies $g(2)=e^{-2}>0$ and
$g'(u)=\tfrac12-e^{-u}\ge\tfrac12-e^{-2}>0$ for $u\ge2$, so $g\ge0$ there.
Combining with \eqref{eq:E1}, for all $x\in[0,1]$,
\[
  \varphi_u(x)\ \ge\ \frac u2\,x(1-x)\qquad(u\ge2).
\]
Since $\sigma(S_c)\subset[0,1]$, the functional calculus turns this pointwise
inequality between continuous functions into the operator inequality
$\varphi_u(S_c)\ge\tfrac u2\bigl(S_c-S_c^2\bigr)\ge0$. Both sides are positive
trace-class operators (Lemma~\ref{lem:upper} for the left side), and the trace
is monotone on positive operators, so
\[
  \Tr\varphi_u(S_c)\ \ge\ \frac u2\,V_c\ \ge\ \frac u2\cdot\frac{\log c}{2\pi^2}
  =\frac{u}{4\pi^2}\log c
\]
by \eqref{eq:A3}.
\end{proof}

\section{A logistic window and its conditional counting consequence}
\label{sec:logistic}

\begin{definition}
\label{def:q}
For $t\in\RR$ put $q(t)=\dfrac1{1+e^t}\in(0,1)$. The map $q$ is a strictly
decreasing real-analytic bijection of $\RR$ onto $(0,1)$. For $t\in\RR$ define
the eigenvalue counting function
\begin{equation}
  N_c(t)=\#\bigl\{j:\lambda_j(c)>q(t)\bigr\},
  \label{eq:Nc}
\end{equation}
eigenvalues being counted with multiplicity. Each $N_c(t)$ is finite because
$S_c$ is trace class with nonnegative eigenvalues, and $N_c$ is nondecreasing
in $t$ because $q$ is decreasing.
\end{definition}

\begin{lemma}[Value of $\varphi_u$ on the logistic window]
\label{lem:E4}
For all $u\ge0$ and $t\in\RR$,
\begin{equation}
  \varphi_u\bigl(q(t)\bigr)
  =\log\frac{e^u+e^t}{1+e^t}-\frac{u}{1+e^t} .
  \label{eq:phiq}
\end{equation}
If $1\le t\le u/2$, then
\begin{equation}
  \varphi_u\bigl(q(t)\bigr)\ \ge\ \frac u2-e^{-1}-\frac{u}{1+e} .
  \label{eq:window1}
\end{equation}
Consequently, for $u\ge8$,
\begin{equation}
  \varphi_u\bigl(q(t)\bigr)\ \ge\ \frac u8,
  \qquad 1\le t\le u/2 .
  \tag{E4}
  \label{eq:E4}
\end{equation}
\end{lemma}

\begin{proof}
\emph{Identity.} With $x=q(t)=(1+e^t)^{-1}$,
\[
  1+(e^u-1)x=\frac{1+e^t+e^u-1}{1+e^t}=\frac{e^u+e^t}{1+e^t} ,
\]
and $ux=u/(1+e^t)$; substituting into \eqref{eq:phi} gives \eqref{eq:phiq}.

\emph{Window bound.} Assume $1\le t\le u/2$ (this range is nonempty only if
$u\ge2$, which holds in the situations where the bound is used). Since
$e^u+e^t\ge e^u$ and $\log$ is increasing,
\[
  \log\frac{e^u+e^t}{1+e^t}\ \ge\ u-\log(1+e^t) .
\]
Now $\log(1+e^t)=t+\log(1+e^{-t})\le t+e^{-t}$ by $\log(1+y)\le y$, and
$t\le u/2$ together with $t\ge1$ gives $\log(1+e^t)\le u/2+e^{-1}$. Hence the
first term of \eqref{eq:phiq} is at least $u/2-e^{-1}$. For the second term,
$t\ge1$ gives $1+e^t\ge1+e$, so $u/(1+e^t)\le u/(1+e)$. Subtracting yields
\eqref{eq:window1}.

\emph{Reduction to $u/8$.} By \eqref{eq:window1} it suffices to show
\[
  \Psi(u):=\frac u2-\frac{u}{1+e}-\frac1e-\frac u8\ \ge\ 0
  \qquad(u\ge8).
\]
Now $\Psi(u)=\gamma u-e^{-1}$ with
$\gamma=\tfrac12-\tfrac1{1+e}-\tfrac18=\tfrac38-\tfrac1{1+e}$. Since
$1+e<8/3$ would be needed for $\gamma\le0$, and in fact $1+e>3.7>8/3$, we have
$\tfrac1{1+e}<\tfrac38$, so $\gamma>0$ and $\Psi$ is strictly increasing. It
therefore suffices to check the endpoint $u=8$:
\[
  \Psi(8)=8\gamma-\frac1e=3-\frac8{1+e}-\frac1e .
\]
Multiplying by the positive number $e(1+e)$,
\[
  e(1+e)\,\Psi(8)=3e(1+e)-8e-(1+e)=3e+3e^2-8e-1-e=3e^2-6e-1 ,
\]
so
\[
  3-\frac8{1+e}-\frac1e>0
  \quad\Longleftrightarrow\quad
  3e^2-6e-1>0 .
\]
The right-hand inequality holds because the larger root of
$3z^2-6z-1$ is $z_+=1+\sqrt{4/3}<2.16<e$. Hence $\Psi(8)>0$ and
$\Psi(u)>0$ for all $u\ge8$, which is \eqref{eq:E4}.
\end{proof}

\begin{lemma}[Exact counting identity]
\label{lem:count}
For $u>2$,
\begin{equation}
  N_c(u/2)-N_c(1)
  =\#\bigl\{j:q(u/2)<\lambda_j(c)\le q(1)\bigr\} .
  \label{eq:countid}
\end{equation}
\end{lemma}

\begin{proof}
Since $u/2>1$ and $q$ is strictly decreasing, $q(u/2)<q(1)$, so
\[
  \bigl\{j:\lambda_j(c)>q(1)\bigr\}\subseteq\bigl\{j:\lambda_j(c)>q(u/2)\bigr\} ,
\]
and both sets are finite. The difference of their cardinalities is the
cardinality of the set-theoretic difference, namely
$\{j:\lambda_j(c)>q(u/2)\}\setminus\{j:\lambda_j(c)>q(1)\}
=\{j:q(u/2)<\lambda_j(c)\le q(1)\}$.
\end{proof}

We now record the first of the two conditional routes to \eqref{eq:target}.
The hypothesis \eqref{eq:LC} below is \emph{not} proved in this note.

\begin{proposition}[Counting hypothesis implies the quadratic bound]
\label{prop:R1}
Let $a>0$, $L>0$, $\alpha>0$ and $c>1$ be given, and suppose that
\begin{equation}
  N_c(u/2)-N_c(1)\ \ge\ a\,u\log c
  \qquad\text{for } L\le u\le\alpha\log c .
  \tag{LC}
  \label{eq:LC}
\end{equation}
Then, for every $u$ with $\max\{L,8\}\le u\le\alpha\log c$,
\begin{equation}
  \Tr\varphi_u(S_c)\ \ge\ \frac a8\,u^2\log c .
  \tag{R1}
  \label{eq:R1}
\end{equation}
\end{proposition}

\begin{proof}
Fix such a $u$; in particular $u\ge8>2$. By the spectral theorem for the
compact self-adjoint operator $S_c$ and Lemma~\ref{lem:upper},
\[
  \Tr\varphi_u(S_c)=\sum_{j\ge1}\varphi_u\bigl(\lambda_j(c)\bigr),
\]
an absolutely convergent sum of \emph{nonnegative} terms, since
$\lambda_j(c)\in[0,1]$ and $\varphi_u\ge0$ there by Lemma~\ref{lem:E1}.
Discarding a subfamily of a series with nonnegative terms only decreases it,
so
\[
  \Tr\varphi_u(S_c)\ \ge\
  \sum_{j\,:\,q(u/2)<\lambda_j(c)\le q(1)}\varphi_u\bigl(\lambda_j(c)\bigr) .
\]
This is precisely the statement that all omitted eigenvalue contributions are
nonnegative and may be dropped.

For each index $j$ in the displayed range, $\lambda_j(c)\in(q(u/2),q(1)]$;
because $q$ is a strictly decreasing bijection of $\RR$ onto $(0,1)$, there is
a unique $t_j\in[1,u/2)$ with $\lambda_j(c)=q(t_j)$. By \eqref{eq:E4},
$\varphi_u(\lambda_j(c))=\varphi_u(q(t_j))\ge u/8$. Hence, using
Lemma~\ref{lem:count} and then \eqref{eq:LC},
\begin{align*}
  \Tr\varphi_u(S_c)
  &\ \ge\ \frac u8\,\#\bigl\{j:q(u/2)<\lambda_j(c)\le q(1)\bigr\}
   =\frac u8\bigl(N_c(u/2)-N_c(1)\bigr)\\
  &\ \ge\ \frac u8\cdot a\,u\log c=\frac a8\,u^2\log c . \qedhere
\end{align*}
\end{proof}

\begin{proposition}[A mean-type counting hypothesis implies \eqref{eq:LC}]
\label{prop:A4}
Let $C\ge0$, $\beta>0$, $c>1$, and suppose that
\begin{equation}
  \Bigl|N_c(t)-c-\frac{t}{\pi^2}\log c\Bigr|\ \le\ C\log c,
  \qquad 1\le t\le\beta\log c .
  \tag{MT}
  \label{eq:MT}
\end{equation}
Set
\begin{equation}
  a=\frac1{4\pi^2},\qquad
  L=4+8\pi^2C,\qquad
  0<\alpha\le2\beta .
  \tag{A4}
  \label{eq:A4}
\end{equation}
Then for every $u$ with $L\le u\le\alpha\log c$ (a range that may be empty, in
which case there is nothing to prove) both $t=1$ and $t=u/2$ lie in the range
allowed by \eqref{eq:MT}, one has
\begin{equation}
  N_c(u/2)-N_c(1)\ \ge\
  \Bigl(\frac{u}{2\pi^2}-\frac1{\pi^2}-2C\Bigr)\log c ,
  \label{eq:MTdiff}
\end{equation}
and \eqref{eq:LC} holds with the constants \eqref{eq:A4}.
\end{proposition}

\begin{proof}
\emph{Admissibility of the two thresholds.}
The value $t=1$ is admissible provided $1\le\beta\log c$, which is implied by
the nonemptiness of the range $L\le u\le\alpha\log c$: indeed then
$\alpha\log c\ge L\ge4$, so $\log c\ge4/\alpha\ge4/(2\beta)=2/\beta$ and
therefore $\beta\log c\ge2>1$. For $t=u/2$ we need $1\le u/2\le\beta\log c$.
The left inequality holds because $u\ge L\ge4$, so $u/2\ge2\ge1$. The right
inequality holds because $u\le\alpha\log c\le2\beta\log c$, so
$u/2\le\beta\log c$. Thus \eqref{eq:MT} may be applied at both thresholds.

\emph{The difference bound.}
Applying \eqref{eq:MT} at $t=u/2$ from below and at $t=1$ from above,
\[
  N_c(u/2)\ \ge\ c+\frac{u}{2\pi^2}\log c-C\log c,
  \qquad
  N_c(1)\ \le\ c+\frac1{\pi^2}\log c+C\log c .
\]
Subtracting, the terms $c$ cancel exactly and we obtain \eqref{eq:MTdiff}.

\emph{Deduction of \eqref{eq:LC}.}
With $a=1/(4\pi^2)$ we have
\[
  \Bigl(\frac{u}{2\pi^2}-\frac1{\pi^2}-2C\Bigr)-a\,u
  =\frac{u}{4\pi^2}-\frac1{\pi^2}-2C
  =\frac1{4\pi^2}\bigl(u-4-8\pi^2C\bigr)
  =\frac{u-L}{4\pi^2}\ \ge\ 0
\]
whenever $u\ge L$. Hence for $L\le u\le\alpha\log c$,
\[
  N_c(u/2)-N_c(1)\ \ge\
  \Bigl(\frac{u}{2\pi^2}-\frac1{\pi^2}-2C\Bigr)\log c\ \ge\ a\,u\log c ,
\]
using $\log c>0$. This is exactly \eqref{eq:LC}.
\end{proof}

\begin{remark}[Logical status of \eqref{eq:LC} and \eqref{eq:MT}]
\label{rem:notproved}
Neither \eqref{eq:LC} nor \eqref{eq:MT} is proved in this note. Both are
hypotheses. Propositions~\ref{prop:R1} and~\ref{prop:A4} are proved
implications: they say that \emph{if} the hypothesis holds for the stated
range of parameters, \emph{then} the stated conclusion holds. Chaining them,
\eqref{eq:MT} with constants $C,\beta$ yields \eqref{eq:LC} with the constants
\eqref{eq:A4}, and \eqref{eq:LC} in turn yields \eqref{eq:R1} for
$\max\{L,8\}\le u\le\alpha\log c$, i.e.\ \eqref{eq:target} with
$\kappa_T=a/8=1/(32\pi^2)$ and $L_T=\max\{4+8\pi^2C,8\}$. Nothing in
Sections~\ref{sec:factorization}--\ref{sec:variance} supplies either
hypothesis.
\end{remark}

\section{The positive moment expansion}
\label{sec:expansion}

\begin{theorem}[Pointwise expansion]
\label{thm:P1}
For every $u\in[0,\infty)$ and $x\in[0,1]$,
\begin{equation}
  \varphi_u(x)=\sum_{n=1}^{\infty}A_n(u)\,x(1-x)^n ,
  \tag{P1}
  \label{eq:P1}
\end{equation}
all terms being nonnegative.
\end{theorem}

\begin{proof}
Fix $u\in[0,\infty)$, so that $r_u:=1-e^{-u}\in[0,1)$; the finiteness of $u$ is
used exactly here.

\emph{Step 1: the derivative as a geometric series.}
Rewrite the denominator in \eqref{eq:dphi}: since
$1-(1-e^{-u})(1-x)=1-(1-e^{-u})+(1-e^{-u})x=e^{-u}+(1-e^{-u})x$, we may write
\begin{equation}
  \partial_u\varphi_u(x)
  =x(1-x)\,\frac{1-e^{-u}}{1-(1-e^{-u})(1-x)} .
  \label{eq:dphigeo}
\end{equation}
For $x\in[0,1]$ and $u<\infty$,
\[
  0\le(1-e^{-u})(1-x)\le1-e^{-u}<1 ,
\]
so the geometric series $\sum_{n\ge0}\bigl[(1-e^{-u})(1-x)\bigr]^n$ converges
to $\bigl[1-(1-e^{-u})(1-x)\bigr]^{-1}$. Multiplying by the nonnegative factor
$x(1-x)(1-e^{-u})$ and reindexing,
\begin{equation}
  \partial_u\varphi_u(x)
  =\sum_{n=1}^{\infty}(1-e^{-u})^n\,x(1-x)^n .
  \label{eq:dphiseries}
\end{equation}

\emph{Step 2: term-by-term integration.}
For $0\le x\le1$ and $t\ge0$ each summand $(1-e^{-t})^nx(1-x)^n$ is
nonnegative and jointly measurable in $(n,t)$ (counting measure times Lebesgue
measure). Tonelli's theorem for nonnegative integrands therefore permits
interchanging $\sum_n$ and $\int_0^u\!\dd t$ with no further justification:
\[
  \varphi_u(x)=\int_0^u\partial_t\varphi_t(x)\dd t
  =\int_0^u\sum_{n=1}^{\infty}(1-e^{-t})^nx(1-x)^n\dd t
  =\sum_{n=1}^{\infty}\Bigl(\int_0^u(1-e^{-t})^n\dd t\Bigr)x(1-x)^n ,
\]
which is \eqref{eq:P1} by the definition \eqref{eq:TnAn} of $A_n(u)$. Here we
again used $\varphi_0(x)=0$ and the continuous differentiability of
$t\mapsto\varphi_t(x)$.

\emph{Step 3: the endpoints.}
At $x=0$ every summand vanishes and $\varphi_u(0)=\log1-0=0$. At $x=1$ every
summand vanishes because of the factor $(1-x)^n$ with $n\ge1$, and
$\varphi_u(1)=\log(1+e^u-1)-u=u-u=0$. So \eqref{eq:P1} holds at the endpoints
as well. Note also that Step~1 is valid at both endpoints: at $x=1$ the
denominator in \eqref{eq:dphigeo} equals $1$ and the numerator vanishes, and at
$x=0$ the denominator equals $e^{-u}>0$ and the numerator again vanishes.
\end{proof}

\begin{proposition}[Elementary bounds on the moments]
\label{prop:Tnbounds}
For every $n\ge1$ and $c>0$ the operator $S_c(I-S_c)^n$ is positive and trace
class, and
\begin{equation}
  0\ \le\ T_n(c)\ \le\ \Tr S_c=c .
  \label{eq:Tnbounds}
\end{equation}
Moreover $T_n(c)=\sum_{j\ge1}\lambda_j(c)\bigl(1-\lambda_j(c)\bigr)^n$.
\end{proposition}

\begin{proof}
The function $f_n(x)=x(1-x)^n$ satisfies $0\le f_n(x)\le x$ on $[0,1]$, since
$0\le(1-x)^n\le1$ there. As $\sigma(S_c)\subset[0,1]$, the functional calculus
gives $0\le S_c(I-S_c)^n\le S_c$ in the operator order; the left inequality
shows positivity and the right one, combined with $\Tr S_c=c<\infty$, shows
trace-class membership together with \eqref{eq:Tnbounds}. The eigenvalue
formula is the spectral theorem applied to the continuous function $f_n$, the
resulting series converging absolutely because its terms are nonnegative with
sum at most $c$.
\end{proof}

\begin{theorem}[Trace expansion]
\label{thm:P2}
For every $u\in[0,\infty)$ and $c>0$,
\begin{equation}
  \Tr\varphi_u(S_c)=\sum_{n=1}^{\infty}A_n(u)\,T_n(c) ,
  \tag{P2}
  \label{eq:P2}
\end{equation}
both sides being finite; indeed
$\Tr\varphi_u(S_c)\le(e^u-1-u)c$ by Lemma~\ref{lem:upper}.
\end{theorem}

\begin{proof}
By Lemma~\ref{lem:upper}, $\varphi_u(S_c)$ is a positive trace-class operator,
so by the spectral theorem
$\Tr\varphi_u(S_c)=\sum_j\varphi_u(\lambda_j(c))$ with all terms nonnegative.
Insert \eqref{eq:P1} at each $x=\lambda_j(c)\in[0,1]$:
\[
  \Tr\varphi_u(S_c)
  =\sum_{j\ge1}\sum_{n\ge1}A_n(u)\,\lambda_j(c)\bigl(1-\lambda_j(c)\bigr)^n .
\]
This is a double series with nonnegative terms, so Tonelli's theorem for
counting measure (equivalently, the monotone convergence theorem applied to
the increasing partial sums) permits interchanging the two summations:
\[
  \Tr\varphi_u(S_c)
  =\sum_{n\ge1}A_n(u)\sum_{j\ge1}\lambda_j(c)\bigl(1-\lambda_j(c)\bigr)^n
  =\sum_{n\ge1}A_n(u)\,T_n(c) ,
\]
by Proposition~\ref{prop:Tnbounds}. The left-hand side is finite by
\eqref{eq:tracefinite}, hence so is the right-hand side.

Equivalently, and without passing through eigenvalues, put
\[
  \Phi_N(x)=\sum_{n=1}^{N}A_n(u)\,x(1-x)^n .
\]
These are continuous on $[0,1]$ and increase pointwise to $\varphi_u$ by
Theorem~\ref{thm:P1}, so $\Phi_N(S_c)$ increases to $\varphi_u(S_c)$ in the
operator order; the trace is normal on increasing nets of positive operators,
so it passes to the limit.
\end{proof}

\section{Fixed-index Landau--Widom asymptotics}
\label{sec:lw}

We now quote one external input, and we are careful to quote it only within
its own scope. We call $F\colon[0,1]\to\RR$ \emph{admissible} when $F$ belongs
to the class of functions for which Landau and Widom \cite{LandauWidom}
establish \eqref{eq:LW}; we do not restate their regularity hypotheses here,
and nothing in this note depends on their precise form, because the only
functions to which \eqref{eq:LW} is ever applied below are the polynomials
$F_n(x)=x(1-x)^n$ with $F_n(0)=0$, which lie in that class. For such $F$ the
integral in \eqref{eq:LW} converges absolutely, since $F(x)-xF(1)$ vanishes at
both $x=0$ and $x=1$ and the quotient by $x(1-x)$ stays bounded.

\begin{theorem}[Landau--Widom, cited; fixed $F$ only]
\label{thm:LW}
For each \emph{one fixed} admissible $F$,
\begin{equation}
  \Tr F(S_c)
  = cF(1)
  +\frac{\log(2\pi c)}{\pi^2}\int_0^1\frac{F(x)-xF(1)}{x(1-x)}\dd x
  +o(\log c)
  \qquad(c\to\infty).
  \tag{LW}
  \label{eq:LW}
\end{equation}
See \cite{LandauWidom}.
\end{theorem}

\begin{mdframed}[linewidth=0.8pt,backgroundcolor=black!3,roundcorner=3pt]
\textbf{Scope of \eqref{eq:LW}.} The remainder $o(\log c)$ is asserted for one
fixed function $F$: it means that there is a function
$\varepsilon_F(c)\to0$ with remainder bounded by $\varepsilon_F(c)\log c$, and
both $\varepsilon_F$ and any threshold $c\ge c_F$ beyond which a given
smallness holds are permitted to depend on $F$ in an arbitrary way. \emph{No
uniformity over families $F=F_n$ is assumed or used anywhere in this note.}
\end{mdframed}

\begin{proposition}[Fixed-$n$ moment asymptotics]
\label{prop:P3}
For each fixed integer $n\ge1$,
\begin{equation}
  T_n(c)=\frac{\log c}{\pi^2n}+o_n(\log c)\qquad(c\to\infty),
  \tag{P3}
  \label{eq:P3}
\end{equation}
where the subscript records that the implied remainder may depend on $n$.
\end{proposition}

\begin{proof}
Fix $n\ge1$ and take $F_n(x)=x(1-x)^n$, a polynomial with $F_n(0)=0$; it is
admissible. Then $T_n(c)=\Tr F_n(S_c)$ by Proposition~\ref{prop:Tnbounds}, and
\[
  F_n(1)=1\cdot0^n=0,
  \qquad
  \frac{F_n(x)-xF_n(1)}{x(1-x)}=\frac{x(1-x)^n}{x(1-x)}=(1-x)^{n-1}
  \quad(0<x<1),
\]
so that
\[
  \int_0^1\frac{F_n(x)-xF_n(1)}{x(1-x)}\dd x=\int_0^1(1-x)^{n-1}\dd x=\frac1n .
\]
Substituting into \eqref{eq:LW} and using
$\log(2\pi c)=\log c+\log(2\pi)$, whose constant contribution
$\log(2\pi)/(\pi^2n)$ is $O_n(1)=o_n(\log c)$, gives
\[
  T_n(c)=0+\frac{\log c}{\pi^2 n}+o_n(\log c) . \qedhere
\]
\end{proof}

\begin{mdframed}[linewidth=1pt,backgroundcolor=black!6,roundcorner=3pt]
\textbf{Fixed-index versus uniform asymptotics.}
Proposition~\ref{prop:P3} is a statement about each fixed $n$ separately. It
gives \emph{no} information whatsoever about $T_{n(c)}(c)$ for a moment index
$n=n(c)$ that grows with $c$, and in particular it does not bound $T_n(c)$
from below uniformly in the range $1\le n\le c^{\beta}$. Both the threshold
$c_n$ beyond which the remainder in \eqref{eq:P3} is small and the size of
that remainder may depend on $n$ arbitrarily badly.
\end{mdframed}

\section{The quantitative moment hypothesis (QM)}
\label{sec:qm}

\begin{unproved}[(QM)]
\label{hyp:QM}
There exist constants $a,\beta,c_0>0$ such that
\begin{equation}
  T_n(c)\ \ge\ a\,\frac{\log c}{n}
  \qquad\text{for all }c\ge c_0\text{ and all integers }1\le n\le c^{\beta} .
  \tag{QM}
  \label{eq:QM}
\end{equation}
\emph{This hypothesis is assumed, never proved, in this note.} It is a
uniform-in-$n$ statement; by the discussion in Section~\ref{sec:lw} it is
\emph{not} a consequence of \eqref{eq:P3}.
\end{unproved}

\begin{definition}[Truncation index]
\label{def:M}
To avoid any confusion with the counting function $N_c(t)$ of
Section~\ref{sec:logistic}, we denote the moment truncation index by
\begin{equation}
  M(u)=\bigl\lceil e^{2u}\bigr\rceil ,
  \label{eq:M}
\end{equation}
so that $e^{2u}\le M(u)\le e^{2u}+1$.
\end{definition}

\begin{lemma}[The harmonically weighted total mass]
\label{lem:harmonic}
For every $u\ge0$,
\begin{equation}
  \sum_{n=1}^{\infty}\frac{A_n(u)}{n}
  =\int_0^u\sum_{n=1}^{\infty}\frac{(1-e^{-t})^n}{n}\dd t
  =\int_0^ut\dd t=\frac{u^2}{2} .
  \label{eq:harmonic}
\end{equation}
\end{lemma}

\begin{proof}
All terms $(1-e^{-t})^n/n$ are nonnegative and jointly measurable, so
Tonelli's theorem justifies exchanging $\sum_n$ with $\int_0^u\dd t$; together
with the definition of $A_n(u)$ this gives the first equality. For the second,
fix $t\in[0,u]$ and set $r_t=1-e^{-t}\in[0,1)$. The Mercator series
$\sum_{n\ge1}z^n/n=-\log(1-z)$ converges for $|z|<1$, so
\[
  \sum_{n=1}^{\infty}\frac{r_t^{\,n}}{n}=-\log(1-r_t)=-\log\bigl(e^{-t}\bigr)=t .
\]
Integrating $t$ over $[0,u]$ gives $u^2/2$.
\end{proof}

\begin{lemma}[Tail estimate and the truncated mass]
\label{lem:P4}
Let $u\ge0$, $0\le t\le u$, and $r_t=1-e^{-t}$. Then
\begin{equation}
  \sum_{n>M(u)}\frac{r_t^{\,n}}{n}
  \ \le\ \frac{r_t^{\,M(u)+1}}{M(u)\,(1-r_t)}
  \ \le\ \frac{e^{t}}{M(u)}
  \ \le\ e^{-u} .
  \label{eq:tail}
\end{equation}
Consequently, for $u\ge2$,
\begin{equation}
  \sum_{n=1}^{M(u)}\frac{A_n(u)}{n}\ \ge\ \frac{u^2}{2}-ue^{-u}\ \ge\ \frac{u^2}{4} .
  \tag{P4}
  \label{eq:P4}
\end{equation}
\end{lemma}

\begin{proof}
\emph{Tail.} Write $M=M(u)\ge1$. For $n>M$ we have $1/n<1/M$, so
\[
  \sum_{n>M}\frac{r_t^{\,n}}{n}
  \le\frac1M\sum_{n>M}r_t^{\,n}
  =\frac1M\cdot\frac{r_t^{\,M+1}}{1-r_t}
  =\frac{r_t^{\,M+1}}{M(1-r_t)} ,
\]
the geometric sum being convergent because $0\le r_t<1$. Since $r_t\le1$ and
$1-r_t=e^{-t}$, the middle quantity is at most $e^{t}/M$. Finally $t\le u$ and
$M=M(u)\ge e^{2u}$ give $e^t/M\le e^u\cdot e^{-2u}=e^{-u}$.

\emph{Truncated mass.} By Lemma~\ref{lem:harmonic} and Tonelli again,
\[
  \sum_{n=1}^{M(u)}\frac{A_n(u)}{n}
  =\int_0^u\sum_{n=1}^{M(u)}\frac{r_t^{\,n}}{n}\dd t
  =\frac{u^2}{2}-\int_0^u\sum_{n>M(u)}\frac{r_t^{\,n}}{n}\dd t
  \ \ge\ \frac{u^2}{2}-\int_0^ue^{-u}\dd t
  =\frac{u^2}{2}-ue^{-u} .
\]
For the last inequality in \eqref{eq:P4} it suffices that
$ue^{-u}\le u^2/4$, i.e.\ $ue^{u}\ge4$; the left side is increasing in $u$ and
equals $2e^2>14>4$ at $u=2$.
\end{proof}

\begin{proposition}[Conditional quadratic bound at a single scale]
\label{prop:EQM}
Assume \eqref{eq:QM} with constants $a,\beta,c_0$. Let $u\ge2$ and $c\ge c_0$
satisfy $M(u)\le c^{\beta}$. Then
\begin{equation}
  \Tr\varphi_u(S_c)\ \ge\ \frac a4\,u^2\log c .
  \tag{E1-QM}
  \label{eq:EQM}
\end{equation}
\end{proposition}

\begin{proof}
By Theorem~\ref{thm:P2} and the nonnegativity of every term
$A_n(u)T_n(c)\ge0$ (Proposition~\ref{prop:Tnbounds}), truncating the series at
$M(u)$ can only decrease it:
\[
  \Tr\varphi_u(S_c)=\sum_{n=1}^{\infty}A_n(u)T_n(c)
  \ \ge\ \sum_{n=1}^{M(u)}A_n(u)T_n(c) .
\]
Every index $n$ in the truncated range satisfies $1\le n\le M(u)\le c^{\beta}$,
and $c\ge c_0$, so \eqref{eq:QM} applies to each of them and gives
$T_n(c)\ge a\log c/n$. Since $A_n(u)\ge0$,
\[
  \Tr\varphi_u(S_c)\ \ge\ a\log c\sum_{n=1}^{M(u)}\frac{A_n(u)}{n}
  \ \ge\ a\log c\cdot\frac{u^2}{4}
\]
by \eqref{eq:P4}, which requires $u\ge2$.
\end{proof}

\begin{lemma}[Admissible truncation range]
\label{lem:range}
Let $0<\alpha<\beta/2$. If $u\le\alpha\log c$ and $c>1$, then
\begin{equation}
  M(u)\le e^{2u}+1\le 2e^{2u}\le2c^{2\alpha} ,
  \label{eq:Mrange}
\end{equation}
and hence $M(u)\le c^{\beta}$ as soon as
$c\ge2^{1/(\beta-2\alpha)}$.
\end{lemma}

\begin{proof}
The first inequality is \eqref{eq:M}. The second holds because $e^{2u}\ge1$
for $u\ge0$. For the third, $u\le\alpha\log c$ gives
$e^{2u}\le e^{2\alpha\log c}=c^{2\alpha}$. Finally,
$2c^{2\alpha}\le c^{\beta}$ is equivalent to $2\le c^{\beta-2\alpha}$; since
$\beta-2\alpha>0$ and $c>1$, this is equivalent to
$\log2\le(\beta-2\alpha)\log c$, i.e.\ to $c\ge2^{1/(\beta-2\alpha)}$.
\end{proof}

\begin{theorem}[Conditional quadratic lower bound]
\label{thm:P5}
Assume \eqref{eq:QM} with constants $a,\beta,c_0$, and let $0<\alpha<\beta/2$.
Then for every
\begin{equation}
  c\ \ge\ \max\bigl\{c_0,\ 2^{1/(\beta-2\alpha)},\ e^{2/\alpha}\bigr\}
  \label{eq:cthreshold}
\end{equation}
one has
\begin{equation}
  2\le u\le\alpha\log c
  \quad\Longrightarrow\quad
  \Tr\varphi_u(S_c)\ \ge\ \frac a4\,u^2\log c .
  \tag{P5}
  \label{eq:P5}
\end{equation}
In particular the range $2\le u\le\alpha\log c$ is nonempty, and \eqref{eq:P5}
is exactly the target \eqref{eq:target} with
\begin{equation}
  \alpha=\frac\beta3,\qquad
  \kappa_T=\frac a4,\qquad
  L_T=2 .
  \label{eq:safechoice}
\end{equation}
\end{theorem}

\begin{proof}
Let $c$ satisfy \eqref{eq:cthreshold}. From $c\ge e^{2/\alpha}$ we get
$\alpha\log c\ge2$, so the range $2\le u\le\alpha\log c$ is nonempty; also
$c>1$. Fix $u$ in that range. By Lemma~\ref{lem:range}, using
$c\ge2^{1/(\beta-2\alpha)}$, we get $M(u)\le c^{\beta}$. Since moreover
$c\ge c_0$ and $u\ge2$, Proposition~\ref{prop:EQM} applies and yields
\eqref{eq:EQM}, which is the assertion \eqref{eq:P5}.

For the final statement, the choice $\alpha=\beta/3$ satisfies
$0<\alpha<\beta/2$, so the above applies verbatim with $\alpha$ pinned to
$\beta/3$; the threshold \eqref{eq:cthreshold} is then to be read with that
same value, namely
$c\ge\max\{c_0,\,2^{3/\beta},\,e^{6/\beta}\}$. Comparing the resulting
conclusion with \eqref{eq:target} we read off $\kappa_T=a/4$ and $L_T=2$.
\end{proof}

\begin{remark}[What Theorem~\ref{thm:P5} does and does not say]
\label{rem:P5status}
Theorem~\ref{thm:P5} is a proved implication whose hypothesis is
\eqref{eq:QM}.  This particular route establishes its conclusion only under
that hypothesis; the unconditional and independent proof of
\eqref{eq:target} is the determinant route of Section~\ref{sec:determinant},
 which uses none of the hypotheses of this section.  The constants
$\kappa_T=a/4$ and $\alpha=\beta/3$ inherit whatever values $a$ and $\beta$
would have in a proof of \eqref{eq:QM}, and $L_T=2$ comes from the
requirement $ue^u\ge4$ in Lemma~\ref{lem:P4}, not from \eqref{eq:QM}.
\end{remark}

\section{Comparison with the determinant method}
\label{sec:determinant}

The target \eqref{eq:target} also follows from a signed
growing-parameter determinant estimate proved in \cite{CompanionC}.
We record the operator identification and the resulting trace estimate
to compare this argument with the operator and moment methods above.

\subsection{Unitary equivalence and determinant identity}\label{sec:dictionary}

We first identify $S_c$ with a sine-kernel operator on a fixed interval.

\begin{proposition}[unitary equivalence and determinant identity]
\label{prop:unitarylink}
For $s>0$ let $K_s$ be the integral operator on $L^2(-1,1)$ with kernel
\begin{equation}\label{eq:Ks}
 K_s(\lambda,\mu)=\frac{\sin\bigl(s(\lambda-\mu)\bigr)}{\pi(\lambda-\mu)},
 \qquad K_s(\lambda,\lambda)=\frac s\pi ,
\end{equation}
and put $s=\pi c/2$.  Define
\begin{equation}\label{eq:Uc}
 (U_cf)(\lambda)=\sqrt{\tfrac c2}\,f\!\left(\tfrac c2(1+\lambda)\right),
 \qquad -1<\lambda<1 .
\end{equation}
Then $U_c:L^2(0,c)\to L^2(-1,1)$ is unitary and
\begin{equation}\label{eq:unitaryequivalence}
 U_cS_cU_c^{-1}=K_{\pi c/2}.
\end{equation}
Consequently, writing
\begin{equation}\label{eq:Ddef}
 D(s,\omega)=\det\bigl(I+(e^{2\omega}-1)K_s\bigr),\qquad\omega\ge0,
\end{equation}
one has, for $u=2\omega$, the exact identity
\begin{equation}
 \Tr\varphi_u(S_c)=\log D\!\left(\frac{\pi c}2,\frac u2\right)-uc .
 \tag{U}
 \label{eq:dettraceidentity}
\end{equation}
\end{proposition}

\begin{proof}
The substitution $x=\tfrac c2(1+\lambda)$ has Jacobian $c/2$ and is a
bijection of $(-1,1)$ onto $(0,c)$, so
$\|U_cf\|_{L^2(-1,1)}^2=\tfrac c2\int_{-1}^1
|f(\tfrac c2(1+\lambda))|^2\dd\lambda=\int_0^c|f(x)|^2\dd x$ and $U_c$
is onto; hence $U_c$ is unitary.  With $x=\tfrac c2(1+\lambda)$ and
$y=\tfrac c2(1+\mu)$ the transformed kernel is
\[
 \frac c2\,K\!\left(\frac c2(\lambda-\mu)\right)
 =\frac c2\cdot
 \frac{\sin\bigl(\tfrac{\pi c}2(\lambda-\mu)\bigr)}
      {\tfrac{\pi c}2(\lambda-\mu)}
 =\frac{\sin\bigl(\tfrac{\pi c}2(\lambda-\mu)\bigr)}{\pi(\lambda-\mu)},
\]
which is $K_{\pi c/2}(\lambda,\mu)$ for $\lambda\ne\mu$; at
$\lambda=\mu$ both sides equal $c/2=s/\pi$, and both kernels are
continuous, so \eqref{eq:unitaryequivalence} holds.

By Theorem~\ref{thm:factorization}, $S_c$ is a positive trace-class
contraction, so $\sigma(S_c)\subset[0,1]$ and $\varphi_u(S_c)$ is trace
class by Lemma~\ref{lem:upper}.  The spectral theorem and the definition
of the Fredholm determinant give
\[
 \Tr\varphi_u(S_c)
 =\sum_j\Bigl[\log\bigl(1+(e^u-1)\lambda_j(c)\bigr)-u\lambda_j(c)\Bigr]
 =\log\det\bigl(I+(e^u-1)S_c\bigr)-u\Tr S_c ,
\]
both series converging absolutely.  Now $\Tr S_c=c$ by \eqref{eq:A1},
and by \eqref{eq:unitaryequivalence} together with the unitary
invariance of the Fredholm determinant,
$\det(I+(e^u-1)S_c)=\det(I+(e^u-1)K_{\pi c/2})=D(\pi c/2,u/2)$ when
$u=2\omega$.  This is \eqref{eq:dettraceidentity}.
\end{proof}

\subsection{The determinant asymptotic}\label{sec:quoteddet}

\begin{quotedthm}[Signed growing-parameter sine-kernel determinant
\cite{CompanionC}]\label{qt:signed}
Let $D(s,\omega)$ be as in \eqref{eq:Ddef}.  For every $A>0$ there are
constants $s_A\ge3$ and $C_A<\infty$, depending only on $A$, such that
for all $s\ge s_A$ and all $\omega$ with $0\le\omega\le A\log s$,
\begin{equation}
 \log D(s,\omega)
 =\frac{4\omega s}{\pi}+\frac{2\omega^2}{\pi^2}\log(4s)
 +2\log\bigl[\BG(1+i\omega/\pi)\BG(1-i\omega/\pi)\bigr]
 +\mathcal R_A(s,\omega),
 \tag{RH}
 \label{eq:RH}
\end{equation}
where $\BG$ is the Barnes $G$-function and
\begin{equation}\label{eq:RHerror}
 |\mathcal R_A(s,\omega)|\le C_A\frac{(1+\omega)^4\log^2s}{s}.
\end{equation}
\end{quotedthm}

\begin{remark}[Parameter range in Theorem~\ref{qt:signed}]
\label{rem:signedscope}
Two warnings.  (i) Setting $\gamma=1-e^{2\omega}\le0$, the determinant in
\eqref{eq:Ddef} is $\det(I-\gamma K_s)$ with $\gamma\le0$.  The
published growing-parameter theorem of Bothner, Deift, Its and Krasovsky
\cite[Thm.~1.2]{BDIK2} gives the same three main terms, with a better
error and on a much longer range $0\le v<s^{1/3}$, but only for
$\gamma\in[0,1)$; it does not cover the sign needed here, and
\eqref{eq:RH} may not be obtained from it by analytic continuation in
$\gamma$.  (ii) The fixed-parameter case of \eqref{eq:RH} is classical
\cite{BasorWidom83,BudylinBuslaev} and is available with an explicit
error from \cite[eq.~(1.4)]{Charlier21}, whose uniformity is on compact
parameter sets only; it therefore gives no information when $\omega$
grows with $s$.  Theorem~\ref{qt:signed} supplies uniformity on the
logarithmic window on the side
$\gamma<0$.
\end{remark}

\subsection{Consequences for the entropy trace}\label{sec:detconseq}

\begin{corollary}[Uniform entropy-trace asymptotic]
\label{cor:traceasymptotic}
For every $A>0$, uniformly
for $0\le u\le2A\log(\pi c/2)$ as $c\to\infty$,
\begin{equation}
 \Tr\varphi_u(S_c)
 =\frac{u^2}{2\pi^2}\log(2\pi c)
 +2\log\!\left[\BG\!\left(1+\frac{iu}{2\pi}\right)
 \BG\!\left(1-\frac{iu}{2\pi}\right)\right]
 +O_A\!\left(\frac{(1+u)^4\log^2c}{c}\right).
 \tag{RH--S}
 \label{eq:RHS}
\end{equation}
\end{corollary}

\begin{proof}
Insert $s=\pi c/2$ and $\omega=u/2$ into \eqref{eq:RH} and subtract
$uc$, using \eqref{eq:dettraceidentity}.  The linear terms cancel
\emph{exactly}, without any asymptotic notation:
\begin{equation}\label{eq:linearcancellation}
 \frac{4\omega s}{\pi}=\frac{4\cdot\frac u2\cdot\frac{\pi c}2}{\pi}=uc .
\end{equation}
Since $4s=2\pi c$, the logarithmic term becomes
$\frac{2(u/2)^2}{\pi^2}\log(2\pi c)=\frac{u^2}{2\pi^2}\log(2\pi c)$, and
since $\frac{i\omega}\pi=\frac{iu}{2\pi}$ the Barnes term is as
displayed.  The error is
$C_A(1+u/2)^4\log^2(\pi c/2)/(\pi c/2)=O_A((1+u)^4\log^2c/c)$.
\end{proof}

\begin{corollary}[the target, via the determinant route]
\label{cor:targetdet}
For every fixed $\alpha>0$
there is $c_T(\alpha)>1$ such that
\begin{equation}
 2\le u\le\alpha\log c
 \quad\Longrightarrow\quad
 \Tr\varphi_u(S_c)\ \ge\ \frac{u^2}{4\pi^2}\log c
 \qquad(c\ge c_T(\alpha)) .
 \tag{T}
 \label{eq:Tproved}
\end{equation}
That is, \eqref{eq:target} holds with
\begin{equation}
 L_T=2,\qquad \kappa_T=\frac1{4\pi^2},\qquad
 \text{every }\alpha>0,
 \tag{T0}
 \label{eq:T0}
\end{equation}
only the threshold $c_T$ depending on $\alpha$.
\end{corollary}

\begin{proof}
This is proved in \cite{CompanionC}; we indicate the two ingredients so
that the constants \eqref{eq:T0} can be checked.  First, a quantitative
lower bound for the Barnes term, proved there from the Weierstrass
product: for $u\ge2$,
\begin{equation}
 2\log\!\left[\BG\!\left(1+\frac{iu}{2\pi}\right)
 \BG\!\left(1-\frac{iu}{2\pi}\right)\right]
 \ \ge\ -\frac{u^2}{\pi^2}\log(2+u).
 \tag{BG}
 \label{eq:Barneslower}
\end{equation}
Second, two uniform smallness statements on $2\le u\le\alpha\log c$:
$\log(2+u)/\log c\le\log(2+\alpha\log c)/\log c\to0$, and, using
$(1+u)^4/u^2\le\tfrac{81}{16}u^2\le\tfrac{81}{16}\alpha^2\log^2c$ for
$u\ge2$,
\[
 \frac{C_A(1+u)^4\log^2c/c}{u^2\log c}
 \le\frac{81C_A\alpha^2}{16}\cdot\frac{\log^3c}{c}\longrightarrow0 .
\]
Taking $A=\alpha/2$ in Corollary~\ref{cor:traceasymptotic} (legitimate
since $\pi/2>1$, so $\alpha\log c\le\alpha\log(\pi c/2)$), and choosing
$c_T(\alpha)$ so large that the first ratio is at most $\tfrac18$ and
the second at most $1/(8\pi^2)$, one gets from
$\log(2\pi c)\ge\log c$ that
\[
 \Tr\varphi_u(S_c)\ \ge\
 \left(\frac12-\frac18-\frac18\right)\frac{u^2}{\pi^2}\log c
 =\frac{u^2}{4\pi^2}\log c .
\]
Both smallness estimates are uniform on the whole window, so one
threshold serves for every $u$ in it.
\end{proof}

\subsection{Relation to the operator and moment methods}
\label{sec:notsuperseded}

Corollary~\ref{cor:targetdet} proves \eqref{eq:target} with a longer
window and better constants than the conditional arguments above.  The
operator-theoretic results nevertheless contain distinct information.

\begin{enumerate}[label=\textup{(\arabic*)},leftmargin=2.4em]
\item The determinant estimate is a nonlinear steepest-descent result
from \cite{CompanionC}.  By contrast,
Sections~\ref{sec:factorization}--\ref{sec:qm} use Fourier analysis,
functional calculus, and fixed-index Landau--Widom asymptotics; without
an additional uniformity hypothesis, these tools give the linear bound
\eqref{eq:E3}.
\item \emph{The hypotheses are of independent interest.}
\eqref{eq:LC} and \eqref{eq:MT} are statements about the prolate
eigenvalue counting function, and \eqref{eq:QM} is a statement about
uniformity in the Landau--Widom moment asymptotics.  Each would, if
proved, yield strictly more than \eqref{eq:target}: (MT) is a
two-sided quantitative counting statement, and (QM) is a uniform-in-$n$
strengthening of \eqref{eq:P3}.  Neither follows from
Corollary~\ref{cor:targetdet}.
\item \emph{Different structural information.}  The exact variance
identity \eqref{eq:A2}, the exact counting identity of
Lemma~\ref{lem:count}, and the positive expansion \eqref{eq:P1} ---
in which \emph{every} term is nonnegative, so that any subfamily may be
discarded --- are exact statements, not asymptotic ones, and are not
consequences of \eqref{eq:RHS}.
\item The results of
Sections~\ref{sec:factorization}--\ref{sec:qm}, including
\eqref{eq:E3} and the three conditional implications, are logically
independent of the determinant estimate.
\end{enumerate}

Thus (LC), (MT), and (QM) remain natural open questions about the
prolate spectrum even though the determinant method establishes the
particular lower bound \eqref{eq:target}.

\section{Assumptions and external inputs}
\label{sec:cert}

The factorization, variance identity, elementary trace bounds, and
positive moment expansion in Sections~\ref{sec:factorization}--
\ref{sec:expansion} are proved directly.  Proposition~\ref{prop:P3} uses
the fixed-test-function theorem of Landau and Widom
(Theorem~\ref{thm:LW}), once for each fixed moment index.  The
determinant comparison in Section~\ref{sec:determinant} uses the signed
growing-parameter asymptotic from \cite{CompanionC}.  These are the two
external asymptotic inputs.

The statements (LC), (MT), and (QM) have a different role: each is an
explicit hypothesis about uniformity in a growing spectral or moment
parameter.  Propositions~\ref{prop:R1} and \ref{prop:A4} and
Theorem~\ref{thm:P5} prove the implications
\[
 (\mathrm{MT})\Longrightarrow(\mathrm{LC})
 \Longrightarrow\eqref{eq:target},
 \qquad
 (\mathrm{QM})\Longrightarrow\eqref{eq:target}.
\]
The hypotheses themselves are not derived here.  They are also not
needed for Corollary~\ref{cor:targetdet}, which follows instead from the
determinant theorem.

The distinction between fixed-index and uniform asymptotics is
important.  For each fixed integer $n\ge1$, Proposition~\ref{prop:P3}
provides a threshold $c_n$ and a remainder
$\varepsilon_n(c)\to0$ such that
\[
 \left|T_n(c)-\frac{\log c}{\pi^2n}\right|
 \le \varepsilon_n(c)\log c
 \qquad(c\ge c_n).
\]
Neither the threshold nor the rate of convergence is controlled as
$n$ varies.  Consequently this statement gives no estimate for a
growing index $n=n(c)$ and does not imply the uniform range
$1\le n\le c^\beta$ in (QM).  Establishing such uniformity is precisely
the additional content of (QM).

\section{Open questions}\label{sec:programme}

We close by collecting three uniformity questions suggested by the
preceding arguments.

\begin{enumerate}[label=\textup{(\arabic*)},leftmargin=2.4em]
\item \textbf{(LC), the counting hypothesis.}  \emph{There are
$a,L,\alpha>0$ such that
$N_c(u/2)-N_c(1)\ge a\,u\log c$ for $L\le u\le\alpha\log c$ and all large
$c$}, with $N_c$ the logistic-window counting function of
Definition~\ref{def:q}.  By Proposition~\ref{prop:R1} this gives
\eqref{eq:target} with $\kappa_T=a/8$ and $L_T=\max\{L,8\}$.  It is a
statement purely about the distribution of the prolate eigenvalues in
the plunge, and it is strictly stronger than what
Corollary~\ref{cor:targetdet} delivers, since the latter controls only a
weighted trace.
\item \textbf{(MT), the mean-type hypothesis.}  \emph{There are
$C\ge0$, $\beta>0$ with
$|N_c(t)-c-t\pi^{-2}\log c|\le C\log c$ for $1\le t\le\beta\log c$ and
all large $c$.}  By Proposition~\ref{prop:A4} this implies (LC) with the
explicit constants \eqref{eq:A4}.  It is a two-sided quantitative form
of the Landau--Widom law, uniform on a logarithmically growing threshold
window.  We are not aware of any proof of it.
\item \textbf{(QM), the uniform moment hypothesis.}  \emph{There are
$a,\beta,c_0>0$ with $T_n(c)\ge a\log c/n$ for all $c\ge c_0$ and all
$1\le n\le c^\beta$.}  By Theorem~\ref{thm:P5} this implies
\eqref{eq:target} with $\kappa_T=a/4$, $L_T=2$, $\alpha=\beta/3$.  It is
precisely the uniform-in-$n$ strengthening of \eqref{eq:P3} that
Theorem~\ref{thm:LW} does not provide.
\end{enumerate}

A proof of any one of (1)--(3) would be of interest independently of
\eqref{eq:target}, and would remain of interest given
Corollary~\ref{cor:targetdet}.  Conversely, no argument in this note,
and no argument known to us, derives any of them from the determinant
route.

\section*{Acknowledgments}

AI-assisted tools were used during manuscript preparation for language
editing, bibliographic cross-checking, and limited symbolic and numerical
consistency checks.  Such checks were used only as supporting verification
and not as substitutes for mathematical proof.  The author reviewed the
resulting manuscript and accepts full responsibility for all statements,
derivations, citations, and conclusions.

\bibliographystyle{plain}
\bibliography{references}

\par\bigskip
\noindent\textsc{Ahmadreza Azimifard}\\
Harmonic Research \& Technologies, LLC\\
\textit{Email:} \href{mailto:afard@harmonicrt.com}{afard@harmonicrt.com}

\end{document}